\documentclass[10pt,a4paper,twoside]{article}

\usepackage{graphicx}
\usepackage{mathptmx}      
\usepackage{latexsym}
\usepackage{amsfonts}
\usepackage{algpseudocode}
\usepackage{algorithm}
\usepackage{epstopdf}
\usepackage{anysize}
\usepackage{caption}
\newcommand{\R}{\mathbb R}
\newcommand{\Z}{\mathbb Z}
\newcommand{\Q}{\mathbb Q}

\newtheorem{theorem}{Theorem}[section]
\newtheorem{lemma}[theorem]{Lemma}

\newtheorem{definition}[theorem]{Definition}

\begin{document}

\title{Projective Splitting Algorithms for Integer Linear Programming\\
{\Large Part 1: Pure Integer Programs}
}

\author{
Federico Rodes\footnote{Licentiate in Applied Mathematics, FCEyN, Universidad de Buenos Aires, Argentina. E-mail: \texttt{rodesf@gmail.com}}
\and 
Isabel Mendez-Diaz\footnote{Departamento de Computacion, FCEyN, Universidad de Buenos Aires, Argentina. Tel.: +54-11-4576-3390, Fax: +54-11-45763359, E-mail: \texttt{imendez@dc.uba.ar}}
\and  
Paula Zabala\footnote{Departamento de Computacion, FCEyN, Universidad de Buenos Aires/CONICET, Argentina. 
Tel.: +54-11-4576-3390, Fax: +54-11-45763359, E-mail: \texttt{pzabala@dc.uba.ar}}
}

\date{}

\maketitle

\begin{abstract}
We propose a new exact approach for solving integer linear programming (ILP) problems which we will call projective splitting algorithms (PSAs). Unlike classical methods for solving ILP problems, PSAs conduct the search for the optimal solution by generating candidate solutions tailored to specific values of the objective function. As a consequence of this strategy, the number of variables in the original ILP problem is systematically reduced without adding any additional constraint to the initial formulation. 

This is the first of a two-part series on PSAs. In this paper we focus on the resolution of pure integer linear programming (PILP) problems, leaving the treatment of mixed integer linear programming (MILP) formulations to the second part of this series. The proposed algorithm was tested against the IBM ILOG CPLEX \cite{CPLEX} optimizer on instances of the 0-1 Multidimensional Knapsack Problem (0-1MKP), showing satisfactory results on instances with a large number of variables.

\end{abstract}


\section{Introduction} 
\label{sec:1}

Linear programming (LP) \cite{D1,D2} is a mathematical modelling technique designed to optimize a linear function (objective function) of non-negative continuous variables (decision variables), while satisfying a system of linear equations or inequalities (constraints). A LP model that restricts some of the variables so that these take only non-negative integer values is known as MILP. When all variables are integer-constrained, we have a PILP model. We will use the term ILP to refer to any of the two types of problems mentioned above.
 
Many practical situations can be modelled as LP problems where decision variables must take on integer values. Generating good timetables, determining optimal schedules for jobs which are to be processed in a production line, designing efficient communication networks, container loading, determining efficient vehicle routes, and various problems arising in computational biology are a few examples.

From a practical point of view, most of the examples mentioned above are extremely difficult to solve. In theoretical computer science, this is captured by the fact that many ILP problems are classified as NP-hard \cite{GaJo} problems. Thus, because of the inherent difficulty and the enormous practical importance of NP-hard ILP problems, a large number of techniques have been proposed to solve them. The available techniques can roughly be classified into two main categories: {\em exact} and {\em heuristic} algorithms. Exact algorithms are guaranteed to find an optimal solution and to prove its optimality for every instance of an ILP problem. The run-time, however, often increases dramatically with the problem instance's size, and often only small or moderately-sized instances can be practically solved to proven optimality. For larger instances, the only possibility is usually to turn to heuristic algorithms that trade optimality for run-time, i.e., they are designed to obtain good but not necessarily optimal solutions in a reasonable amount of time.

The aim of this paper is to propose---to the best of our knowledge---a new exact algorithm for solving PILP problems. The algorithm will be called PSA-{\em pilp}, and the idea behind it is to decompose the initial PILP problem into simpler one-dimensional subproblems, and then to use that information to generate a finite number of candidate solutions tailored to each of the possible optimal objective values of the problem. The optimal solution is then found by examining the set of candidate solutions arising from the previous analysis. The second part of this series is intended to extend this methodology to the class of MILP formulations.

The remainder of this paper is organized as follows. In the next section, we give a short overview of ILP techniques and introduce some notation. In Section~\ref{sec:3}, we present the basic concepts involving the PSA-{\em pilp} algorithm and account for the main steps of the method through the solution of a simple example. Section~\ref{sec:4} is devoted to present the scheme of the PSA-{\em pilp} algorithm and to prove its convergence. Computational results on instances of the 0-1MKP are reported in Section~\ref{sec:5}. Finally, Section~\ref{sec:6} summarizes the main achievements of the proposed approach and outlines some interesting directions for future research.


\section{Integer linear programming, an overview} 
\label{sec:2}

This section gives a short overview of the main concepts in integer programming. For an in-depth coverage of the subject we refer to the books on linear optimization by Chv\'atal \cite{Chvatal}, and on combinatorial and integer optimization by Wolsey \cite{Wo} and Nemhauser and Wolsey \cite{Nem}.

\subsection{LP and ILP formulations} 
\label{subsec:2.1}

LP is a class of optimization problems that involves non-negative continuous variables, an objective function linearly depending on the variables, and a set of constraints expressed as linear inequalities. We consider the form
\begin{eqnarray} \label{def:lp}
({\bf LP}) \ \ \ \
\mathrm{maximize} \ \ 
z(\mathbf{x}) = \mathbf{c}^\mathrm{T}\mathbf{x} + h \nonumber \\
\mathrm{subject \ to} \ \  
\mathbf{A}\mathbf{x} \leq \mathbf{b}\\
\mathbf{x} \in \R_+^n \nonumber
\end{eqnarray}
where $\mathbf{c} \in \Z^{n}$, $h \in \Z$, $\mathbf{A} \in \Q^{n \times m}$ and $\mathbf{b} \in \Q^{m}$ are data. A {\em feasible solution} to (\ref{def:lp}) is a vector $\mathbf{x} \in \R_+^n$ satisfying the condition $\mathbf{A}\mathbf{x} \leq \mathbf{b}$. The aim of this problem is to find a feasible solution that maximizes the objective function $z(\mathbf{x})$.

As mentioned before, if we restrict some of the variables of a LP problem to take on integer values we obtain an ILP problem. We consider the form
\begin{eqnarray} \label{def:ilp}
({\bf ILP}) \ \ \ \ 
\mathrm{maximize} \ \ 
z(\mathbf{x}) = \mathbf{c}^\mathrm{T}\mathbf{x} + h \nonumber \\
\mathrm{subject \ to} \ \
\mathbf{A}\mathbf{x} \leq \mathbf{b}\\
\mathbf{x} \in \Z_+^p \times \R_+^{n-p} \nonumber
\end{eqnarray}
where $\mathbf{c}$, $h$, $\mathbf{A}$ and $\mathbf{b}$ are defined as in (\ref{def:lp}). Without loss of generality, we assume that the variables indexed $1$ through~$p$, $p \leq n$, are the integer-constrained variables (the integer variables), and the variables indexed $p + 1$ through $n$ are called the continuous variables.

Throughout this work it will be assumed for simplicity that the feasible regions of (\ref{def:lp}) and (\ref{def:ilp}) are bounded. In addition, we will denote by $z_{\bf LP}$ (resp. $z_{\bf ILP}$) the optimal objective value of the problem, and by $cod(z(\mathbf{x}))$ the codomain of $z(\mathbf{x})$ for the problem under consideration.  
Finally, let us note that, in the context of PILP problems, the assumption made about the objective function (we do not loss generality) automatically implies that $cod(z(\mathbf{x})) \subseteq \Z$. The utility of this observation will become clear in Section~\ref{sec:3}.

\subsection{LP-relaxation} 
\label{subsec:2.2}

One of the most important concepts in ILP are {\em relaxations}, where some or all constraints of a problem are loosened or omitted. Relaxations are mostly used to obtain related, simpler problems which can be solved efficiently yielding bounds for the original problem.

The {\em linear programming relaxation} of the ILP problem (\ref{def:ilp}) is obtained by relaxing the integrality constraint, i.e., replacing $\mathbf{x} \in \Z_+^p \times \R_+^{n-p}$ with $\mathbf{x} \in \R_+^n$, yielding the LP problem (\ref{def:lp}). Large instances of such LP problems can be efficiently solved in practice by using simplex-based algorithms \cite{D1,D2}, interior-point methods \cite{KAR} or column generation approaches \cite{Chvatal}. As the feasible points of an ILP problem form a subset of the feasible points of its LP-relaxation, the optimal value of the LP-relaxation provides an upper bound on the optimal value of the original ILP problem. Therefore, if an optimal solution to the LP-relaxation satisfies the integrality restrictions, then that solution is also optimal for the ILP problem.

\subsection{Exact algorithms} 
\label{subsec:2.3}

When considering exact approaches, the following methods have had significant success. See e.g. \cite{Gom,LD,Wo} for a general introduction to these mathematical programming techniques.

\subsubsection*{Cutting plane approach} 
\label{subsubsec:2.3.1}

When modelling integer optimization problems as ILP problems, an important goal is to find a strong formulation, for which the LP-relaxation provides a solution which lies in general not too far away from the integer optimum. For many such problems it is possible to strengthen an existing ILP formulation significantly by including further inequalities, preferably, facets of the convex hull of feasible solutions. 

The general cutting plane approach relaxes initially the integrality restrictions of the original ILP problem and solves the resulting linear program. In case the resulting LP solution satisfies the integer requirements, this is the solution to the integer program; otherwise, the LP-relaxation can be tightened by adding an extra constraint which is satisfied by all feasible integral solutions but is violated by the current LP optimal solution. Such a constraint is called a {\em cut} or {\em cutting plane}. The new LP-relaxation is then resolved, and the procedure can be repeated until an optimal solution is reached. The subproblem of identifying cuts is called {\em separation problem}, and it is of crucial importance to solve it efficiently, since many instances of it must usually be solved until the cutting plane approach terminates successfully.

\begin{center}
\begin{tabbing} 
\hspace{1cm}\=\hspace{1cm}\=\hspace{1cm}\=\hspace{1cm}\=\hspace{1cm}\=\hspace{1cm}\= \\
\rule[0.1cm]{10cm}{0.01cm}\\
{\bf Algorithm 1} The Generic \texttt{Cutting-Plane} Algorithm\\
\rule[0.1cm]{10cm}{0.01cm}\\ 
{\bf Input:} $({\bf ILP}) \ \ \mathrm{max} \ \ z(\mathbf{x}) = \mathbf{c}^\mathrm{T}\mathbf{x} + h \ \ \mathrm{s.t.} \ \ \mathbf{A}\mathbf{x} \leq \mathbf{b}, \ \mathbf{x} \in \Z_+^p \times \R_+^{n-p}$\\
{\bf repeat}\\
\> solve the LP-relaxation of {\bf ILP}. Let $\mathbf{x^*}$ be an optimal solution.\\
\> {\bf if} $\mathbf{x^*}$ satisfies the integrality requirements {\bf then}\\
\> \> an optimal solution to {\bf ILP} has been found. {\bf stop}.\\
\> {\bf else}\\
\> \> solve the separation problem, that is, try to find a valid inequality $\mathbf{w}^\mathrm{T}\mathbf{x} \leq d$ such that $\mathbf{w}^\mathrm{T}\mathbf{x^*} > d$.\\
\> \> {\bf if} such an inequality $\mathbf{w}^\mathrm{T}\mathbf{x} \leq d$ cutting off $\mathbf{x^*}$ was found {\bf then}\\
\> \> \> add the inequality to the system.\\
\> \> {\bf else}\\
\> \> \> no optimal solution to {\bf ILP} was found. {\bf stop}.\\
\> \> {\bf end if}\\
\> {\bf end if}\\
{\bf until} forever
\end{tabbing}
\end{center}
\vspace*{-0.2cm}
\rule[0.1cm]{10cm}{0.01cm}
\vspace*{0.3cm}

In practice, it may take a long time for such a cutting plane approach to converge to the optimum, partly because it is often a hard subproblem to separate effective cuts. A further drawback of this technique is that no feasible integer solutions can be obtained until the optimal integer solution is reached, which implies that there is no feasible integer solution if the computations are stopped prematurely. The cutting plane method is therefore often combined with other methods, as we will see below.

\subsubsection*{Branch-and-bound methods} 
\label{subsubsec:2.3.2}

The basic structure of branch-and-bound is an {\em enumeration tree}. The {\em root} node of the tree corresponds to the original problem. As the algorithm progresses, the tree grows by a process called {\em branching}, which creates two or more child nodes of the parent node. Each of the problems at the child nodes is formed by adding constraints to the problem at the parent node. Typically, the new constraint is obtained by simply adding a bound on a single integer variable, where one child gets an upper bound of some integer $d$, and the other child gets a lower bound of $d + 1$. An essential requirement is that each feasible solution to the parent node problem is feasible to at least one of the child node problems.

Let ${\bf ILP}(0)$ be the original ILP problem and let ${\bf ILP}(k)$ be the problem at node $k$. The objective value of any feasible solution to ${\bf ILP}(k)$ provides a lower bound on the global optimal value. The feasible solution with the highest objective value found so far is called the {\em incumbent} solution and its objective value is denoted by $z^{best}$. Let $\mathbf{x}^k$ be an optimal solution to the LP-relaxation of ${\bf ILP}(k)$ with objective value $z^k$. If $\mathbf{x}^k$ satisfies the integrality constraints, then it is an optimal solution to ${\bf ILP}(k)$ and a feasible solution to ${\bf ILP}(0)$, and therefore we update $z^{best}$ as $\mathrm{max} \{z^k , z^{best}\}$. Otherwise, there are two possibilities: if $z^k \leq z^{best}$, then an optimal solution to ${\bf ILP}(k)$ cannot improve on $z^{best}$, hence the subproblem ${\bf ILP}(k)$ is removed from consideration; on the other hand, if $z^k > z^{best}$, then ${\bf ILP}(k)$ requires further exploration, which is done by {\em branching}, i.e., by creating $q \geq 2$ new subproblems ${\bf ILP}(k(i))$, $i = 1, 2, \dots, q$, of ${\bf ILP}(k)$. Each feasible solution to ${\bf ILP}(k)$ must be feasible to at least one child and, conversely, each feasible solution to a child must be feasible to ${\bf ILP}(k)$. Moreover, the solution $\mathbf{x}^k$ must not be feasible to any of the LP-relaxations of the children. A simple realization of these requirements is to select a variable $x_j$ for which $x_j^k$ is not integer and to create two subproblems; in one subproblem, we add the constraint $x_j \leq \lfloor x_j^k \rfloor$, which is the round down of $x_j^k$, and in the other $x_j \geq \lfloor x_j^k \rfloor$, which is the round up of $x_j^k$. The child nodes of node $k$ corresponding to these subproblems are then added to the tree. The largest among all LP-relaxation values associated with the active subproblems provides a global upper bound on the optimal value. The algorithm terminates when the global upper bound and global lower bound ($z^{best}$) are equal.

\begin{center}
\begin{tabbing} 
\hspace{1cm}\=\hspace{1cm}\=\hspace{1cm}\=\hspace{1cm}\=\hspace{1cm}\=\hspace{1cm}\= \\
\rule[0.1cm]{10cm}{0.01cm}\\
{\bf Algorithm 2} The \texttt{Branch-and-Bound} Algorithm\\
\rule[0.1cm]{10cm}{0.01cm}\\
{\bf Input:} $({\bf ILP}) \ \ \mathrm{max} \ \ z(\mathbf{x}) = \mathbf{c}^\mathrm{T}\mathbf{x} + h \ \ \mathrm{s.t.} \ \ \mathbf{A}\mathbf{x} \leq \mathbf{b}, \ \mathbf{x} \in \Z_+^p \times \R_+^{n-p}$.\\
{\bf 0. Initialize.} \\
\> Create a list {\bf L} of active subproblems. Set ${\bf L}=\{ {\bf ILP}(0) \}$,  $z^{best} = -\infty$ and $\mathbf{x}^{best} = \emptyset$.\\
{\bf 1. Terminate?}\\
\> Is ${\bf L} = \emptyset$? If so, {\bf return} $\mathbf{x}^{best}$ is an optimal solution to {\bf ILP}.\\
{\bf 2. Select.}\\
\> Choose and delete a problem ${\bf ILP}(k)$ from {\bf L}.\\
{\bf 3. Evaluate.}\\
\> Solve the LP-relaxation ${\bf LP}(k)$ of ${\bf ILP}(k)$. If ${\bf LP}(k)$ is infeasible, {\bf goto} Step 1,\\ 
\> else let $z^k$ be its objective function value and $\mathbf{x}^k$ be its solution.\\
{\bf 4. Prune.}\\
\> If $z^k \leq z^{best}$, {\bf goto} Step 1. If $\mathbf{x}^k$ is not integer, {\bf goto} Step 5,\\
\> else let $z^{best} = z^k$, $\mathbf{x}^{best} = \mathbf{x}^k$. {\bf Goto} Step 1.\\
{\bf 5. Branch.}\\
\> Divide the feasible domain $S^k$ of ${\bf ILP}(k)$ into smaller sets $S^{k(i)}$ for $i = 1, \dots, q$,\\
\> such that $\cup_{i=1}^q S^{k(i)} = S^k$, and add the subproblems ${\bf ILP}(k(i))$, $i = 1, \dots, q$, to {\bf L}.\\ 
\> {\bf Goto} Step 1.
\end{tabbing}
\end{center}
\vspace*{-0.2cm}
\rule[0.1cm]{10cm}{0.01cm}
\vspace*{0.3cm}

This basic scheme does not specify the rule to follow for choosing a node from {\bf L}. A popular method to do this is to select the node with the highest value $z^k$. Such strategy is known as {\em best-bound search} (or {\em best-first search}). This node selection strategy focuses the search on decreasing the global upper bound, because the only way to decrease the global upper bound is to improve the LP-relaxation at a node with the
highest LP-relaxation value. Another well-known method of selecting a node to explore is to always choose the most recently created node. This is known as {\em diving search} (or {\em depth-first search}). This node selection strategy focuses the search on increasing the global lower bound, because feasible solutions are typically found deep in the tree. In addition to a different focus, both methods also have different computational attributes. Diving search has low memory requirements, because only the sibling nodes on the path to the root of the tree have to be stored. Furthermore, the changes in the linear programs from one node to the next are minimal, a single bound of a variable changes, which allows warm-starts in the LP solves. Best-bound search, on the other hand, favors exploring nodes at the top of the tree as these are more likely to have high LP-relaxation values. This, however, can lead to large list of active subproblems. Furthermore, subsequent linear programs have little relation to each other leading to longer solution times. 

\newpage
We say that node $k$ is {\em superfluous} if $z^k < z_{\bf ILP}$. Best-bound search ensures that no superfluous nodes will be explored. On the other hand, diving search can lead to the exploration of many superfluous nodes that would have been fathomed, had we known a smaller $z^{best}$.

Most integer-programming solvers employ a hybrid of best-bound search and diving search, trying to benefit from the strengths of both, and switch regularly between the two strategies during the search. In the beginning the emphasis is usually more on diving, to find high quality solutions quickly, whereas in the later stages of the search, the emphasis is usually more on best-bound, to improve the global upper bound.\\

Combining branch-and-bound with cutting plane algorithms yields the highly effective class of {\em branch-and-cut} algorithms which are widely used in commercial ILP-solvers such as CPLEX and Gurobi \cite{Gurobi}. Cuts are generated at the nodes of the branch-and-bound search tree to tighten the bounds of the LP-relaxations or to exclude infeasible solutions.


\section{The PSA-{\em pilp} algorithm} 
\label{sec:3}

In Section \ref{sec:2} we carry out a review of the main algorithms employed in the resolution of ILP problems. In all cases, we have seen that the strategy for finding the optimal solution consists of modifying the problem domain (having previously considered its relaxation) through the addition of new constraints. In the case of the cutting planes algorithms, the new inequalities are used to separate fractional solutions of the LP-relaxation and to keep the set of integer solutions of the original ILP problem. In the case of the branch-and-bound and related methods, the inequalities are used for partitioning the problem domain and eliminating fractional solutions of the LP-relaxation.

With a different approach, in this section we present the PSA-{\em pilp} algorithm, which does not alter the problem domain and, consequently, avoids the addition of new constraints to the original formulation.  

\subsection{Definitions and terminology}
\label{subsec:3.1}

Let us begin this section by introducing the concepts of {\em projection}, {\em level} and {\em range} needed to describe the PSA-{\em pilp} algorithm. To this end, consider the two-variable PILP problem illustrated in figure \ref{fig:1} where: (i) it is supposed that the problem is in the form (\ref{def:ilp}); (ii) the set of integer solutions is represented as black points on the $(x_1,x_2)$ plane; and (iii) ${\bf P}_{\bf j}$, $j=1,2$, denotes the {\em projection with respect to the variable $x_j$}, i.e., the shadow cast by $z(\mathbf{x})$ on the $(x_j,z)$ plane.

From figure \ref{fig:1}, it can be observed that the projection ${\bf P}_{\bf j}$ is defined on the interval $[l_j,u_j]$, where the endpoints of this interval clearly represent the minimum and maximum values of the variable $x_j$ over the feasible domain of the LP-relaxation of the problem being solved. Thus, $l_j$ and $u_j$ can be formally defined as follows:  
\[
l_j = z_{\bf LP}, \ \mathrm{with} \ \ ({\bf LP}) \ \ \mathrm{min} \ z(\mathbf{x}) = x_j \ \ \mathrm{s.t.} \ \ \mathbf{A}\mathbf{x} \leq \mathbf{b}, \ \mathbf{x} \in \R_+^2; 
\]
\[
u_j = z_{\bf LP}, \ \mathrm{with} \ \ ({\bf LP}) \ \ \mathrm{max} \ z(\mathbf{x}) = x_j \ \ \mathrm{s.t.} \ \ \mathbf{A}\mathbf{x} \leq \mathbf{b}, \ \mathbf{x} \in \R_+^2.  
\]

The projection ${\bf P}_{\bf j}$ can then be described as the two-dimensional convex set enclosed by the curves $P_j^{low}(x_j): [l_j,u_j] \rightarrow \R$ and $P_j^{up}(x_j): [l_j,u_j] \rightarrow \R$. The former function corresponds to the lower boundary of the set, which we will call the {\em lower projection}, and the latter corresponds to the upper boundary, which we will  call the {\em upper projection}. It is straightforward to see that these curves can be calculated, for each fixed value $x_j= \lambda_j$, by solving two LP problems of one variable: 
\[
P_j^{low}(\lambda_j) = z_{\bf LP}, \ \mathrm{with} \ \ ({\bf LP}) \ \ \mathrm{min} \ z(\mathbf{x}) = \mathbf{c}^\mathrm{T}\mathbf{x} + h \ \ \mathrm{s.t.} \ \ x_j = \lambda_j, \ \mathbf{A}\mathbf{x} \leq \mathbf{b}, \ \mathbf{x} \in \R_+^2;
\]
\[
P_j^{up}(\lambda_j) = z_{\bf LP}, \ \mathrm{with} \ \ ({\bf LP}) \ \ \mathrm{max} \ z(\mathbf{x}) = \mathbf{c}^\mathrm{T}\mathbf{x} + h \ \ \mathrm{s.t.} \ \ x_j = \lambda_j, \ \mathbf{A}\mathbf{x} \leq \mathbf{b}, \ \mathbf{x} \in \R_+^2.
\]

Consequently, for a two-variable problem, the projection of $z(\mathbf{x})$ onto the $(x_j,z)$ plane can be defined as follows:
\[
{\bf P}_{\bf j} :=  \bigr\{ \ (x_j,z) \in \R^2 \ : \ x_j \in [l_j,u_j], \ P_j^{low}(x_j)\leq z \leq P_j^{up}(x_j) \ \bigr\}.
\]

\newpage
\begin{figure}[!h]
\begin{center}
\includegraphics[scale=0.9]{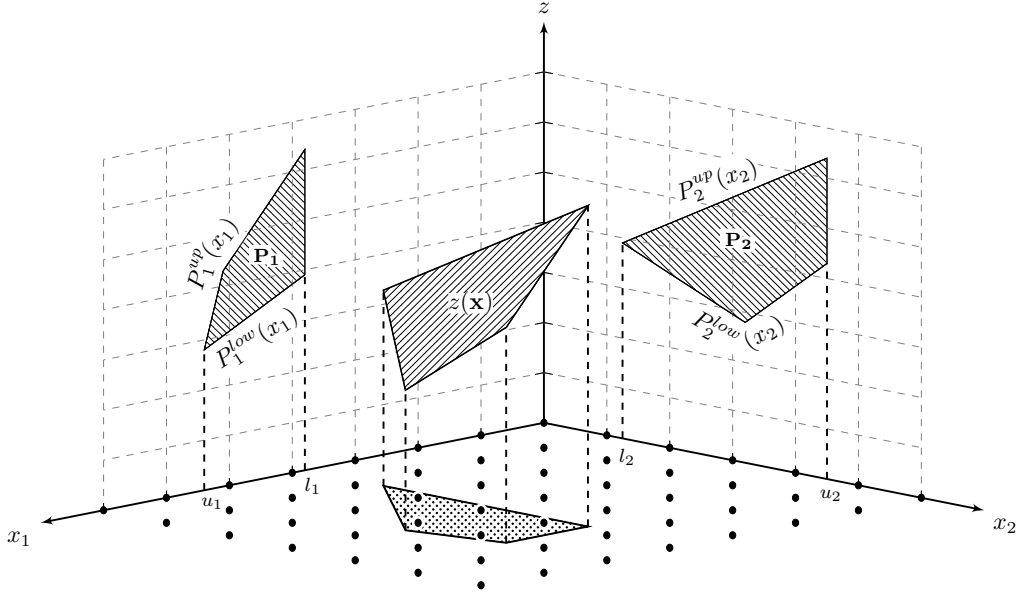}
\caption{orthogonal projections for a two-variable PILP problem}
\label{fig:1} 
\end{center}
\end{figure}

The definition of {\em projection} to be used in this paper is the natural extension of the model introduced above adapted to higher dimensions. 

\begin{definition}[Projection]
Given a PILP problem, for $j=1, \dots,n$ we define the {\em projection} of $z(\mathbf{x})$ onto the $(x_j,z)$ plane, ${\bf P}_{\bf j}$ for short, as the two-dimensional convex set satisfying the following conditions. 
\begin{equation}\label{def:proj}
{\bf P}_{\bf j} :=  \bigr\{ \ (x_j,z) \in \R^2 \ : \ x_j \in [l_j,u_j], \ P_j^{low}(x_j)\leq z \leq P_j^{up}(x_j) \ \bigr\}, 
\end{equation}
where 
\begin{equation}\label{def:lj}
l_j := z_{\bf LP}, \ \mathrm{with} \ \ ({\bf LP}) \ \ \mathrm{min} \ z(\mathbf{x}) = x_j \ \ \mathrm{s.t.} \ \ \mathbf{A}\mathbf{x} \leq \mathbf{b}, \ \mathbf{x} \in \R_+^n; 
\end{equation}
\begin{equation}\label{def:uj}
u_j := z_{\bf LP}, \ \mathrm{with} \ \ ({\bf LP}) \ \ \mathrm{max} \ z(\mathbf{x}) = x_j \ \ \mathrm{s.t.} \ \ \mathbf{A}\mathbf{x} \leq \mathbf{b}, \ \mathbf{x} \in \R_+^n; 
\end{equation}
and where the lower and upper projections, $P_j^{low}(x_j)$ and $P_j^{up}(x_j)$, can be determined, for each fixed value $x_j = \lambda_j \in [l_j,u_j]$, by solving two LP problems of $n-1$ variables:
\begin{equation}\label{def:plow}
P_j^{low}(\lambda_j) = z_{\bf LP}, \ \mathrm{with} \ \ ({\bf LP}) \ \ \mathrm{min} \ z(\mathbf{x}) = \mathbf{c}^\mathrm{T}\mathbf{x} + h \ \ \mathrm{s.t.} \ \ x_j = \lambda_j, \ \mathbf{A}\mathbf{x} \leq \mathbf{b}, \ \mathbf{x} \in \R_+^n; 
\end{equation}
\begin{equation}\label{def:pup}
P_j^{up}(\lambda_j) = z_{\bf LP}, \ \mathrm{with} \ \ ({\bf LP}) \ \ \mathrm{max} \ z(\mathbf{x}) = \mathbf{c}^\mathrm{T}\mathbf{x} + h \ \ \mathrm{s.t.} \ \ x_j = \lambda_j, \ \mathbf{A}\mathbf{x} \leq \mathbf{b}, \ \mathbf{x} \in \R_+^n. 
\end{equation}
\end{definition}

Let us now introduce the concepts of {\em level} and {\em range} which will be used to interpret the information given by the projections. 

\begin{definition}[Level]
Given a PILP problem, we will call {\em level} to each of the values that may be reached by the objective function $z(\mathbf{x})$. More precisely, we will call level to each of the elements of the $cod(z(\mathbf{x}))$ set.
\end{definition}

\begin{definition}[Range]
Given a PILP problem, the set of integer values that can be assigned to the variable $x_j$, $j=1, \dots,n$, when the projection ${\bf P}_{\bf j}$ is restricted to level $z$, will be called the {\em range} of $x_j$ on level $z$. This set will be denoted by $Range_j^{z}$, and a more formal definition is given by: 
\begin{equation}\label{def:range}
Range_j^{z} = \bigr\{ \ r \in \Z : (r, z) \in {\bf P}_{\bf j} \ \bigr\}. 
\end{equation}
\end{definition}

To fix ideas, reconsider the projections ${\bf P}_{\bf 1}$ and ${\bf P}_{\bf 2}$ of the PILP problem shown in figure \ref{fig:1}. Given that $cod(z(\mathbf{x})) \subseteq \Z$, it is straightforward to see that the set of values that may be reached by the objective function is given by: $cod(z(\mathbf{x})) = \{3,4,5,6\}$. The figure presented below illustrates the two largest elements of this set along with the range of integer values that can be assigned to the variables $x_1$ and $x_2$ at each of those levels.

\begin{figure}[!h]
\begin{center}
\includegraphics[scale=1]{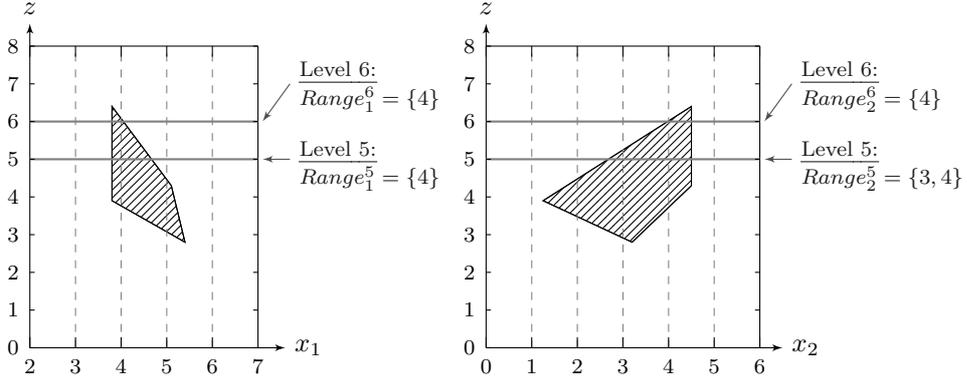}
\caption{projections ${\bf P}_{\bf 1}$ and ${\bf P}_{\bf 2}$ of the previous example crossed by levels $6$ and $5$}
\label{fig:2} 
\end{center}
\end{figure}

\subsection{Motivation} 
\label{subsec:3.2}

Let us now explain the main steps involved in the PSA-{\em pilp} algorithm through the solution of the following instance of the classical {\em Unbounded Knapsack Problem} (UKP). This example covers all possibilities that may occur when applying PSAs for solving PILP problems. 
\begin{eqnarray*}
({\bf UKP}) \ \ \ \ \mathrm{maximize} \ \ z(x_1,x_2,x_3) = 9x_1 + 3x_2 + 8x_3 \\
\mathrm{subject \ to} \ \ 10x_1 + 5x_2 + 7x_3 \leq 12\\
\mathbf{x} = (x_1,x_2,x_3) \in \Z^3_+
\end{eqnarray*}

As stated in \cite{MT90}, the LP-relaxation of every instance of the UKP can be trivially solved by comparing the quotients $\frac{c_j}{a_{1j}}$ corresponding to each variable $x_j$. For this reason, the projections ${\bf P}_{\bf j}$ in this example can be exactly computed by simply applying the expressions (\ref{def:proj}) to (\ref{def:pup}) to the proposed formulation. Thus, the family of projections of {\bf UKP} turns out to be (see figure \ref{fig:3}):
\[
{\bf P}_{\bf 1} =  \Bigr\{ \ (x_1,z) \in \R^2 : x_1 \in \Bigr[0,\frac{12}{10} \Bigr], \ 9 x_1 \leq z \leq -\frac{17}{7} x_1 + \frac{96}{7} \ \Bigr\},
\]
\[
{\bf P}_{\bf 2} =  \Bigr\{ \ (x_2,z) \in \R^2 : x_2 \in \Bigr[0,\frac{12}{5} \Bigr], \ 3 x_2 \leq z \leq -\frac{19}{7} x_2 + \frac{96}{7} \ \Bigr\},
\]
\[
{\bf P}_{\bf 3} =  \Bigr\{ \ (x_3,z) \in \R^2 : x_3 \in \Bigr[0,\frac{12}{7} \Bigr], \ 8 x_3 \leq z \leq \ \frac{17}{10} x_3 + \frac{108}{10} \ \Bigr\}.
\]

These projections make it possible to decompose the original problem into single-variable subproblems, and thus they allow us to study the behaviour of the objective function from each variable's point of view independently. In particular, every time a specific value of the objective function is observed (think of a horizontal line across ${\bf P}_{\bf 1}$, ${\bf P}_{\bf 2}$ and ${\bf P}_{\bf 3}$), the information given by the projections can be used to restrict the {\em range} of integer values that can be assigned to each variable $x_j$. As a result, candidate solutions capable of reaching the selected $z$-value can be generated by combining the allowed values of each of the $\mathbf{x}$-coordinates. 

Given that the set of all possible optimal objective values of {\bf UKP} is finite, namely $cod(z(\mathbf{x})) = \{0,1, \dots,12,13\}$, it becomes natural to address the solution of {\bf UKP} by studying the candidate solutions produced by applying the observation made above to each of the elements of the $cod(z(\mathbf{x}))$ set. Furthermore, because we are maximizing, we can conduct the search process for the optimal solution by considering, one by one in decreasing order of value, each of the elements of the $cod(z(\mathbf{x}))$ set. Then it is easy to see that, if a {\em feasible} candidate solution $\bar{\mathbf{x}}$ satisfying the condition $z(\bar{\mathbf{x}}) = z_i$ is found when level $z_i \in cod(z(\mathbf{x}))$ is being observed, this automatically implies that $\bar{\mathbf{x}}$ is a global optimum to the proposed problem.

\newpage
To show more clearly what we are saying, reconsider the projections ${\bf P}_{\bf 1}$, ${\bf P}_{\bf 2}$ and ${\bf P}_{\bf 3}$ of the problem at hand together with the three largest elements of the $cod(z(\mathbf{x}))$ set. Figure \ref{fig:3} illustrates this situation along with the range of integer values that can be assigned to the variables $x_1$, $x_2$ and $x_3$ at each of those levels. 

\begin{figure}[!h]
\begin{center}
\includegraphics[scale=0.9]{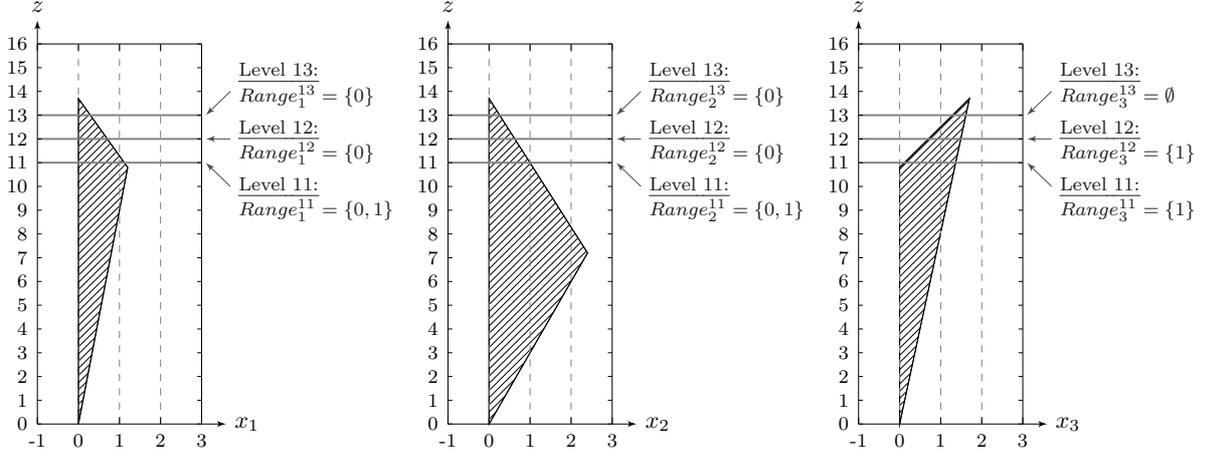}
\caption{projections ${\bf P}_{\bf 1}$, ${\bf P}_{\bf 2}$ and ${\bf P}_{\bf 3}$ of {\bf UKP} crossed by levels 13, 12, and 11}
\label{fig:3}
\end{center}
\end{figure}

{\bf Level 13.} Based on the values contained in the sets $Range_j^{13}$, $j = 1, 2, 3$, it can be inferred that there is no candidate solution capable of reaching level $13$, i.e., we can conclude that $z_{\bf UKP} < 13$. The level $13$ is then discarded from the list of possible optimal objective values of {\bf UKP}, and the search process is continued at level~$12$. 

{\bf Level 12.} From the information given by the sets $Range_j^{12}$, $j = 1, 2, 3$, it can be inferred that $\bar{\mathbf{x}}=(0,0,1)$ is the {\em unique} candidate solution capable of reaching level $12$. Then, to determine whether $12$ is the optimal objective value of the problem and $\bar{\mathbf{x}}=(0,0,1)$ is the associated optimal solution, we simply check the following two conditions (from now on the {\em stopping criterion}) on our candidate solution: (i) is $\bar{\mathbf{x}}$ feasible? (ii) does $z(\bar{\mathbf{x}})=12$? If~the answer to both questions is affirmative, clearly $\bar{\mathbf{x}}=(0,0,1)$ is an optimal solution to the proposed problem and $12$ is the optimal objective value; otherwise, level $12$ is discarded from the list of possible optimal objective values of {\bf UKP}, and the search process is continued at level~$11$. 

By a simple calculation, it is easy to see that $\bar{\mathbf{x}}=(0,0,1)$ is a feasible solution to {\bf UKP}, however, it yields an objective value of $8$. Hence, given that $\bar{\mathbf{x}}=(0,0,1)$ is the unique candidate solution arising from this value of the maximand, it can be concluded that:  $z_{\bf UKP} < 12$ and $z_{\bf UKP} \geq 8$. Before considering the next level and continuing with the search process, it is necessary to introduce two new variables in order to keep the former information: $\bar{\mathbf{x}}^{best}:=(0,0,1)$ ({\em incumbent solution}); $z^{best}:=8$ ({\em lower bound}).

{\bf Level 11.} From the sets $Range_j^{11}$, $j = 1, 2, 3$, it can be inferred that the points $\bar{\mathbf{x}}=(0,0,1)$, $\bar{\mathbf{x}}=(0,1,1)$, $\bar{\mathbf{x}}=(1,0,1)$ and $\bar{\mathbf{x}}=(1,1,1)$ are the {\em only} four candidates for level~$11$. We can now proceed in two different ways in order to determine whether some of these candidates is, in fact, an optimal solution to the proposed problem. The first alternative is to repeat what was done at the previous level, i.e., to simply check the stopping criterion on each of the four candidate solutions. If we are thinking of extending the procedure to higher dimensions, this approach is clearly inefficient due to the exponential growth in the number of candidates. The second alternative, which is the one we are going to use, is to try to extract a little more of the information contained in ${\bf P}_{\bf 1}$, ${\bf P}_{\bf 2}$ and ${\bf P}_{\bf 3}$ in order to reduce the number of candidate solutions arising from the level being scanned. 

With this latter goal in mind, we begin by observing that the condition $|Range_3^{11}|=1$ implies that all possible candidate solutions for the current level must be in the form $\bar{\mathbf{x}}=(?,?,1)$ (such a point will be called a {\em partial candidate solution} to {\bf UKP}). Then, the restriction $x_3 = 1$ can be imposed on the original formulation, thus obtaining a new problem of a smaller dimension. Hereafter, the resulting problem will be called the {\em reduced problem}, and we will denote it by ${\bf UKP}|_{\bar{\mathbf{x}}}$. In our case, the reduced problem turns out to be:
\[
({\bf UKP}|_{\bar{\mathbf{x}}}) \ \ \mathrm{max} \ z(x_1,x_2,1) = 9x_1 + 3x_2 + 8 \ \ \mathrm{s.t.} \ \ 10x_1 + 5x_2 \leq 5, \ (x_1,x_2) \in \Z^2_+.
\]

Now, the general procedure can be applied to the reduced problem: recalculate the projections ${\bf P}_{\bf 1}$ and ${\bf P}_{\bf 2}$, and re-examine level $11$ in order to determine the new sets $Range_1^{11}$ and $Range_2^{11}$.

Before recalculating ${\bf P}_{\bf 1}$ and ${\bf P}_{\bf 2}$ explicitly and carrying on with the example, let us open a parenthesis here to enumerate the 4 alternatives that may hold depending on the cardinality of the new sets $Range_1^{11}$ and $Range_2^{11}$. We will also explain how to proceed in each situation. For convenience in the exposition, the set consisting of the variables that have not yet been fixed will be called {\em active variables} ({\bf AV}). In our case, ${\bf AV} = \{x_1,x_2\}$. 

\begin{enumerate}
\item \label{case:1} {\bf if $\mathbf{|Range_j^{11}| = 1}$ for all $\mathbf{j}$ such that $\mathbf{x_j \in}$ AV}. This means that there exist values $a,b \in \Z$ such that $Range_1^{11}=\{a\}$ and $Range_2^{11}=\{b\}$. Then, we can assert that, if there existed a feasible solution for this value of the maximand, it should be in the form $\bar{\mathbf{x}}=(a,b,1)$. The search process finishes if the resulting point satisfies the stopping criterion. Otherwise, given that $\bar{\mathbf{x}}=(a,b,1)$ is the only candidate solution arising from this level, we can conclude that $z_{\bf UKP} < 11$. In the latter case, before proceeding to the next level and continuing with the search process, we first check whether the variables $\bar{\mathbf{x}}^{best}$ and $z^{best}$ can be updated.

\item \label{case:2} {\bf if $\mathbf{|Range_j^{11}| = 0}$ for at least one $\mathbf{j}$ such that $\mathbf{x_j \in}$ AV}. In this case, there is no integer value that can be assigned to, at least, one of the non-fixed coordinates of $\bar{\mathbf{x}}=(?,?,1)$. Therefore, we can conclude that $z_{\bf UKP} < 11$. Then, the original problem is reconsidered and the search process is restarted at level $10$. 

\item \label{case:3} {\bf if $\mathbf{|Range_j^{11}| = 1}$ for at least one $\mathbf{j}$ such that $\mathbf{x_j \in}$ AV (but not all)}. Without loss of generality, let us suppose that $|Range_1^{11}| = 1$, i.e., there exists a value $a \in \Z$ such that $Range_1^{11}=\{a\}$. Then, the partial candidate solution, the set of active variables, and the reduced problem can be updated as follows: 
\[
\bar{\mathbf{x}}=(a,?,1), \ {\bf AV} =\{x_2\}, \ \mathrm{and} \
({\bf UKP}|_{\bar{\mathbf{x}}}) \ \ \mathrm{max} \ z(a,x_2,1) = 3x_2 + 8 + 9a \ \ \mathrm{s.t.} \ \ 5x_2 \leq 5 - 10a, \ x_2 \in \Z_+.
\]
In this way, the original problem is further reduced in size, and the process can be continued (at the current level) by recalculating the projection ${\bf P}_{\bf 2}$ of ${\bf UKP}|_{\bar{\mathbf{x}}}$, and by performing the same four-step analysis that is being used here.
  
\item \label{case:4} {\bf if $\mathbf{|Range_j^{11}| > 1}$ for all $\mathbf{j}$ such that $\mathbf{x_j \in}$  AV}. In this case, we proceed in the following manner. Firstly, we choose one of the active variables of the problem using some criterion, say $x_s$, and create new partial candidate solutions by assigning the $r^{th}$ value contained in the set $Range_s^{11}$, $1\leq r \leq |Range_s^{11}|$, to the $s^{th}$ component of $\bar{\mathbf{x}}$. By abuse of notation, we will also write $\bar{\mathbf{x}}$ to denote the new partial candidate solutions created in this manner. Secondly, we add all the partial candidate solutions constructed in the previous step to the set of partial candidate solutions to be analysed ({\bf L}). Thirdly, using some criterion, we extract one of the partial candidate solutions added to {\bf L}, say $\bar{\mathbf{x}}$, and calculate the reduced problem associated to it (${\bf UKP}|_{\bar{\mathbf{x}}}$) and redefine {\bf AV} as the set of non-fixed components of $\bar{\mathbf{x}}$. Finally, the search process is continued by recalculating the sets ${\bf P}_{\bf j}$ and $Range_j^{11}$ of ${\bf UKP}|_{\bar{\mathbf{x}}}$ (for the variable $x_j$ that has not yet been fixed), and by performing the same four-step analysis that is being used here. Note that, if the current problem does not produce any optimal solution, it is necessary to analyse the solution space generated by the remaining partial candidate solutions contained in {\bf L} before concluding that $11$ is not the optimal level of $z(\mathbf{x})$. If any optimal solution is reached, the procedure terminates; otherwise, the original problem is reconsidered and the search process is restarted at level $10$.
\end{enumerate}    

Having established the 4 alternatives that may hold depending on the cardinality of the new sets $Range_1^{11}$ and $Range_2^{11}$, let us now come back to the example. In our case, the projections ${\bf P}_{\bf 1}$ and ${\bf P}_{\bf 2}$ of the reduced problem turns out to be (see figure \ref{fig:4}):
\[
{\bf P}_{\bf 1} =  \Bigr\{ \ (x_1,z) \in \R^2 \ : \ x_1 \in \Bigr[0,\frac{1}{2} \Bigr], \ 9x_1 + 8 \leq z \leq \frac{15}{5} x_1 + \frac{55}{5} \ \Bigr\},
\]
\[
{\bf P}_{\bf 2} =  \Bigr\{ \ (x_2,z) \in \R^2 \ : \ x_2\in [0,1], \ 3 x_2 + 8 \leq z \leq -\frac{15}{10} x_2 + \frac{125}{10} \ \Bigr\}.
\]

\newpage
\begin{figure}[!h]
\begin{center}
\includegraphics[scale=0.9]{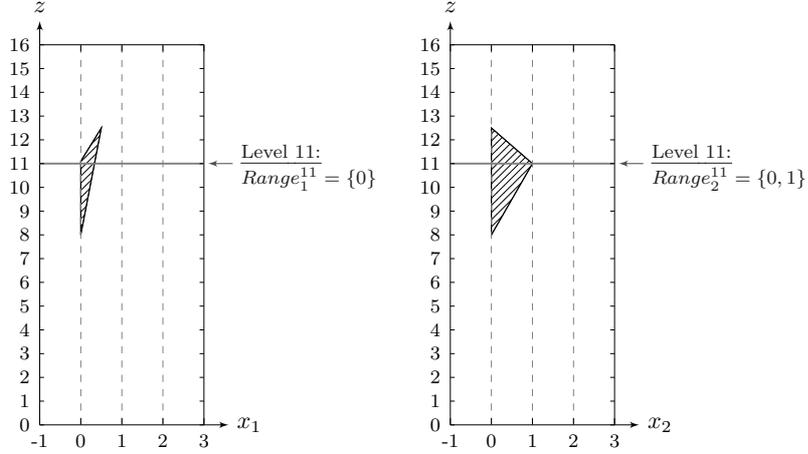}
\caption{projections ${\bf P}_{\bf 1}$ and ${\bf P}_{\bf 2}$ of ${\bf UKP}|_{(?,?,1)}$ crossed by level 11}
\label{fig:4}
\end{center}
\end{figure}

From figure \ref{fig:4}, it can be observed that the range of integer values that can be assigned to each of the remaining active variables is given by: $Range_1^{11}=\{0\}$ and $Range_2^{11}=\{0,1\}$ (case \ref{case:3}). Therefore, the partial candidate solution, the set of active variables, and the reduced problem can be updated as follows: 
\[
\bar{\mathbf{x}}=(0,?,1), \ {\bf AV}=\{x_2\}, \ \mathrm{and} \ ({\bf UKP}|_{\bar{\mathbf{x}}}) \ \ \mathrm{max} \ z(0,x_2,1) = 3x_2 + 8 \ \ \mathrm{s.t.} \ \ 5x_2 \leq 5, \ x_2 \in \Z_+.
\]

The projection ${\bf P}_{\bf 2}$ of ${\bf UKP}|_{\bar{\mathbf{x}}}$ is then recalculated in an attempt to obtain tighter bounds for the set $Range_2^{11}$ (see figure \ref{fig:5}):
\[
{\bf P}_{\bf 2} =  \Bigr\{ \ (x_2,z) \in \R^2 \ : \ x_2 \in [0,1], \ 3 x_2 + 8 \leq z \leq 3 x_2 + 8 \ \Bigr\}.
\]

\begin{figure}[!h]
\begin{center}
\includegraphics[scale=0.9]{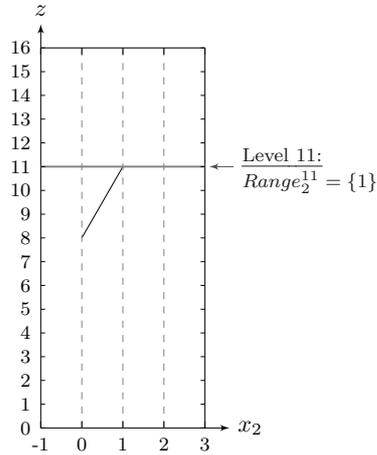}
\caption{projection ${\bf P}_{\bf 2}$ of ${\bf UKP}|_{(0,?,1)}$ crossed by level 11}
\label{fig:5}
\end{center}
\end{figure}

From figure \ref{fig:5} it follows that: $Range_2^{11}=\{1\}$ (case \ref{case:1}). This means that $\bar{\mathbf{x}}=(0,1,1)$ is the unique candidate solution capable of reaching level $11$. Then, given that $\bar{\mathbf{x}}$ satisfies the stopping criterion, we can conclude that it is an optimal solution to {\bf UKP}. $\lhd$ \\

The algorithm to be described in this paper is the generalization of the above procedure adapted to higher dimensions. The following outline summarizes how the proposed algorithm works. Given a PILP problem, the PSA-{\em pilp} algorithm starts by calculating the orthogonal projection associated to each variable $x_j$, and by identifying the set of all possible optimal objective values, say $cod(z(\mathbf{x}))=\{z_{1}, \dots, z_{k}\}$. Then, it begins the search for the optimal solution by considering, one by one in decreasing order of value, each of the elements of the $cod(z(\mathbf{x}))$ set. Every time a new level $z_{i} \in cod(z(\mathbf{x}))$ is selected, the algorithm utilizes the information contained in the sets $Range_j^{z_{i}}$ to fix the value of some of the variables, and thus to reduce the size of the original problem. The procedure is then continued by recalculating the sets ${\bf P}_{\bf j}$ and $Range_j^{z_{i}}$ of the reduced problem for the variables that have not yet been fixed ({\em active variables}). A number of candidate solutions is constructed for each considered level $z_i$ by applying this argument systematically.

The search process ends (stopping criterion) when a {\em feasible candidate solution $\bar{\mathbf{x}}$ satisfying the condition $z(\bar{\mathbf{x}})=z_i$ is found when the algorithm is scanning level $z_i$}. Then, it can be concluded that $\bar{\mathbf{x}}$ is a global optimum to the proposed PILP problem.\\

As can be seen in this outline, and also in the previous example, the proposed algorithm differs from the state-of-the-art techniques in three aspects: (i) it guides the search for the optimal solution by generating candidate solutions tailored to specific values of the objective function; (ii) it systematically reduces the number of variables in the original problem for each considered level; and (iii) it does not add any additional constraint to the initial formulation. Concerning the second point, it is worth noting that, while in the case of branch-and-bound-based algorithms the number of variables that can be fixed in each iteration of the procedure (for each node in the search tree) oscillates between $0$ and $1$, in the case of the PSA-{\em pilp} algorithm this figure ranges between $1$ and $|{\bf AV}|$. Furthermore, as we will see later on in Section \ref{sec:5}, the computational experiments performed on instances of the 0-1MKP reveal that the percentage of variables that are fixed to their optimal value in the {\em first iteration} of the PSA-{\em pilp} algorithm at the optimal level, rise to more than $97\%$ of the total variables.

\subsection{Practical aspects}
\label{subsec:3.3}

To conclude this section, let us give some precisions about how to calculate projections in the case of general PILP problems. This is motivated by the fact that, unlike what happened in Section \ref{subsec:3.2} for the UKP, in the case of general PILP problems it is usually too expensive---or even impossible---to derive explicit formulas for $P_j^{up}(x_j)$ and $P_j^{low}(x_j)$ for all $x_j$ in the domain of the definition of ${\bf P}_{\bf j}$. It then becomes necessary to identify which part of the information provided by the projections is dispensable and which part is strictly necessary for executing the PSA-{\em pilp} algorithm. 

It is easy to see that the only information that is absolutely necessary for executing the PSA-{\em pilp} algorithm is that given by the points $(e_j, P_j^{up}(e_j))$ and $(e_j, P_j^{low}(e_j))$, where $e_j$ takes on all possible {\em integer} values in the domain of the definition of ${\bf P}_{\bf j}$. From a theoretical point of view, this observation makes it possible to compute the set of projections for every instance of a PILP problem in a finite number of steps. In practice, however, it may take a long time for the PSA-{\em pilp} algorithm to converge to the optimum if the coefficients $P_j^{up}(e_j)$ and $P_j^{low}(e_j)$ are calculated exactly. It is then natural to, in addition to the previous simplification, approximate some of these values in order to reduce the number of operations even further. 
 
To fix ideas, the following outline details the steps of the procedure suggested above applied to the computation of the set of upper projections, $\{ P_j^{up}(e_j) \ | \ e_j$ {\it integer in the domain of the definition of} ${\bf P}_{\bf j} \}$, for the subclass of PILP problems in which all variables are restricted to be $0$ or $1$. This type of problems is known as {\em binary integer linear programming} (BILP). A similar approach can be applied to determine the lower projections of a BILP problem as well as to calculate the upper and lower projections for more complex PILP problems.

\begin{itemize}
\item {\em Phase 1.} The integer requirements of the original BILP problem are relaxed and the associated maximization LP program is solved by using the Simplex method. Let $\mathbf{x}^*$ denote the optimal solution to the LP-relaxation, and let $z_{\bf LP}$ denote its respective optimal objective value. It is easy to see that, if the $j^{th}$ component of $\mathbf{x}^*$ yields an integer value $e_j \in \{0,1\}$, this automatically implies that $P_j^{up}(e_j) = z_{\bf LP}$. In other words, assuming that the problem we are trying to solve had $n$ variables, this first operation would allow us to calculate, in the best-case scenario, up to $n-1$ of the total $2n$ coefficients $P_j^{up}(e_j)$, $e_j \in \{0,1\}$.
 
\item {\em Phase 2.} For each of the remaining values, $P_j^{up}(e_j)$, that were not able to be computed in the previous phase, the additional constraint $x_j~\leq~0$ (if $e_j=0$) or $x_j \geq 1$ (if $e_j=1$) is added to the bottom of the optimal Simplex tableau obtained in the previous step, and the dual Simplex algorithm is then used to restore primal feasibility and to compute an upper bound for $P_j^{up}(e_j)$.
\end{itemize}


\section{Scheme and correctness of the PSA-{\em pilp} algorithm} 
\label{sec:4}

In this section, we present the scheme of the algorithm. In order to do that, we assume that there exists a procedure that permits to compute the set of projections for every instance of a PILP problem in a finite number of operations. The same assumption will be made on Section \ref{subsec:4.2} to prove the finiteness and the correctness of the algorithm.

\subsection{Scheme of the algorithm}
\label{subsec:4.1}

As we mentioned before, the strategy of the PSA-{\em pilp} algorithm is to sweep across the set of all possible optimal objective values of the problem, say $cod(z(\mathbf{x}))=\{z_1, \dots, z_k\}$, and to use the information given by the sets $Range_j^{z_i}$, $j$ such that $x_j \in {\bf AV}$, to generate a finite number of candidate solutions tailored to each of the selected $z$-values. The search process finishes when a candidate solution which meets the stopping condition is found. 

In order to clarify the exposition of the algorithm, we will divide the procedure into two parts, thus introducing a slight modification in comparison to the example presented in Section \ref{subsec:3.2}. On the one hand, we will introduce the \texttt{Inspect\_Level} algorithm, which is the responsible for generating the whole set of candidate solutions associated to a given level. On the other hand, we will present the \texttt{Main} algorithm, which is the responsible for performing the parallel shifts in the functional value in the direction of a reduction of the maximand, and for checking the stopping criterion on the set of candidate solutions provided by the \texttt{Inspect\_Level} algorithm. The scheme of the algorithms is as follows:

\begin{center}
\begin{tabbing} 
\hspace{1cm}\=\hspace{1cm}\=\hspace{1cm}\=\hspace{1cm}\=\hspace{1cm}\=\hspace{1cm}\=\hspace{1cm}\= \\
\rule[0.1cm]{10cm}{0.01cm}\\
{\bf Algorithm 3} The \texttt{Main} Algorithm\\
\rule[0.1cm]{10cm}{0.01cm}\\
{\bf Input:} $({\bf PILP}) \ \ \mathrm{max} \ z(\mathbf{x}) = \mathbf{c}^\mathrm{T}\mathbf{x} + h \ \ \mathrm{s.t.} \ \ \mathbf{A}\mathbf{x} \leq \mathbf{b}, \ \mathbf{x} \in \Z_+^n$\\
\\
{\bf Assumption:} $\mathbf{c} \in \Z^n, h \in \Z$\\
\\
{\bf Output:} optimal solution to {\bf PILP}, or detects infeasibility\\
\\
{\bf Variables:}\\ 
$\bar{\mathbf{x}} = (\bar{x}_1, \dots, \bar{x}_n) \in \Z^n$ (candidate solution)\\
$cod(z(\mathbf{x})) \subseteq \Z$ (codomain of $z(\mathbf{x})$ over the feasible domain of {\bf PILP})\\
$z \in  \Z$ (level being scanned)\\
$z^{best}\in \Z$ (lower bound)\\
$\bar{\mathbf{x}}^{best} \in \Z^n$ (incumbent solution)\\
${\bf CS}_z^{\bf PILP}$ (set of candidate solutions to {\bf PILP} produced by the \texttt{Inspect\_Level} algorithm at level $z$)\\
${\bf P}_{\bf j}$ (projection produced by $z(\mathbf{x})$ onto the $(x_j,z)$ plane)\\
\\
{\bf 0. Initialize.} \\
\> compute ${\bf P}_{\bf j}$ for $j =1, \ldots, n$\\
\> compute $cod(z(\mathbf{x}))$ \\
\> set $z$ to the largest element in $cod(z(\mathbf{x}))$\\
\> set $z^{best}$ to the smallest element in $cod(z(\mathbf{x}))$\\
\\
{\bf 1. Loop.}\\
\> {\bf while} $z > z^{best}$ {\bf or} ($z == z^{best}$ {\bf and} $\bar{\mathbf{x}}^{best} == \mathrm{NULL}$) {\bf do}\\
\> \> {\bf 1.1. Inspection.} \\
\> \> \> set ${\bf CS}_z^{\bf PILP} = \texttt{Inspect\_Level} \big( \ {\bf PILP}, z, \big\{ {\bf P}_{\bf j} \big\}_{j=1, \dots, n} \ \big)$\\
\\
\> \> {\bf 1.2. Check.}\\
\> \> \> {\bf for all} $\bar{\mathbf{x}} \in {\bf CS}_z^{\bf PILP}$ {\bf do}  \\
\> \> \> \> {\bf if} $\bar{\mathbf{x}}$ is a feasible solution to {\bf PILP} {\bf and} $z(\bar{\mathbf{x}})== z$\\
\> \> \> \> \> {\bf return }$\bar{\mathbf{x}}$ is an optimal solution to {\bf PILP}\\
\> \> \> \> {\bf else if} $\bar{\mathbf{x}}$ is feasible {\bf and} $z(\bar{\mathbf{x}}) > z^{best}$\\
\> \> \> \> \> set $z^{best}= z(\bar{\mathbf{x}})$\\
\> \> \> \> \> set $\bar{\mathbf{x}}^{best} = \bar{\mathbf{x}}$\\
\> \> \> \> {\bf end if}\\
\> \> \> {\bf end for}\\
\\
\> \> {\bf 1.3. Next level.}\\
\> \> \> set $z = z - 1$\\
\> {\bf end while}\\
\\
{\bf 2. Output.}\\
\> {\bf if} $\bar{\mathbf{x}}^{best} \ne \mathrm{NULL}$\\
\> \> {\bf return} $\bar{\mathbf{x}}^{best}$ is an optimal solution to {\bf PILP}\\
\> {\bf else}\\
\> \> {\bf return} {\bf PILP} is infeasible\\
\> {\bf end if}
\end{tabbing}
\end{center}
\rule[0.1cm]{10cm}{0.01cm}

\begin{center}
\begin{tabbing} 
\hspace{1cm}\=\hspace{1cm}\=\hspace{1cm}\=\hspace{1cm}\=\hspace{1cm}\=\hspace{1cm}\= \\
\rule[0.1cm]{10cm}{0.01cm}\\
{\bf Algorithm 4} The \texttt{Inspect\_Level} Algorithm\\
\rule[0.1cm]{10cm}{0.01cm}\\
{\bf Input:} {\bf PILP} problem, $z$ (level to be scanned), $\big\{ {\bf P}_{\bf j} \big\}_{j=1, \dots, n}$ (set of projections associated to {\bf PILP})\\
\\
{\bf Output:} ${\bf CS}_z^{\bf PILP}$ (set of candidate solutions to {\bf PILP} arising from level $z$)\\
\\
{\bf Variables:} \\
$\bar{\mathbf{x}} = (\bar{x}_1, \dots, \bar{x}_n) \in \Z^n$ (partial candidate solution)\\
{\bf Prob} (problem being analysed)\\
{\bf AV} (set of variables that have not yet been fixed)\\
$\big\{ {\bf P}_{\bf j} \big\}_{j \, : \, x_j \in {\bf AV}}$ (set of projections associated to {\bf Prob})\\
{\bf L} (set of partial candidate solutions to be analysed)\\
${\bf CS}_z^{\bf PILP}$ (set of candidate solutions to {\bf PILP} arising from level $z$)\\
\\
{\bf 0. Initialize.} \\
\> set ${\bf Prob} = {\bf PILP}$\\
\> set ${\bf AV} = \{x_1, \dots, x_n\}$\\
\> set $\big\{ {\bf P}_{\bf j} \big\}_{j \, : \, x_j \in {\bf AV}} = \big\{ {\bf P}_{\bf j} \big\}_{j=1, \dots, n}$\\
\> set ${\bf CS}_z^{\bf PILP} = \{\}$\\
\> set $\bar{x}_j = \mathrm{NULL}$ for all $j=1, \dots,n$\\
\> set ${\bf L} = \{\}$\\
\\
{\bf 1. Inspection.} \\
\> compute $Range_j^{z}$ for all $j$ such that $x_j \in {\bf AV}$\\
\> {\bf if} $|Range_j^{z}|>0$ for all $j$ such that $x_j \in {\bf AV}$ \\
\> \> {\bf if} $\exists j$ such that $x_j \in {\bf AV}$ {\bf and} $|Range_j^{z}|==1$\\
\> \> \> {\bf for all} $j$ such that $x_j \in {\bf AV}$ {\bf and} $|Range_j^{z}|==1$ {\bf do} {\scriptsize \hspace{0.5cm}/*$Range_j^{z}=\{r_j\}$*/}\\
\> \> \> \> set $\bar{x}_j=r_j$\\
\> \> \> \> set ${\bf AV} = {\bf AV} - \{x_j\}$\\
\> \> \> {\bf end for }\\
\> \> \> {\bf if} ${\bf AV} == \emptyset$\\
\> \> \> \> set ${\bf CS}_z^{\bf PILP} = {\bf CS}_z^{\bf PILP} \cup \{\bar{\mathbf{x}}\}$\\
\> \> \> {\bf else}\\
\> \> \> \> set ${\bf Prob} = {\bf PILP}|_{\bar{\mathbf{x}}}$ \\ 
\> \> \> \> compute the set of projections associated to {\bf Prob}: $\big\{ {\bf P}_{\bf j} \big\}_{j \, : \, x_j \in {\bf AV}}$\\
\> \> \> \> {\bf goto} Step 1\\
\> \> \> {\bf end if}\\
\> \> {\bf else } {\scriptsize \hspace{0.5cm}/*$|Range_j^{z}| > 1 \, \forall \, j \ / \ x_j \in {\bf AV}$*/}\\
\> \> \> choose $j$ such that $x_j \in {\bf AV}$ using some criterion {\scriptsize \hspace{0.5cm}/*$Range_j^{z}=\{r_1,\ldots,r_{|Range_j^{z}|}\}$*/}\\
\> \> \> {\bf if} $|{\bf AV}|==1$\\
\> \> \> \> {\bf for} $i=1$ {\bf to} $|Range_j^{z}|$ {\bf do}\\
\> \> \> \> \> set $\bar{x}_j = r_i$\\
\> \> \> \> \> set ${\bf CS}_z^{\bf PILP} = {\bf CS}_z^{\bf PILP} \cup \{\bar{\mathbf{x}}\}$\\
\> \> \> \> {\bf end for }\\
\> \> \> {\bf else}\\
\> \> \> \> {\bf for } $i=1$ {\bf to} $|Range_j^{z}|$ {\bf do}\\
\> \> \> \> \> set $\bar{x}_j = r_i$\\
\> \> \> \> \> set ${\bf L} = {\bf L} \cup \{\bar{\mathbf{x}}\}$\\
\> \> \> \> {\bf end for }\\
\> \> \> {\bf end if}\\
\> \> {\bf end if}\\
\> {\bf end if} \\
\\
{\bf 2. Update.} \\
\> {\bf if} ${\bf L} \ne \emptyset $ \\
\> \> choose $\bar{\mathbf{x}} \in {\bf L}$ using some criterion\\
\> \> set ${\bf L} = {\bf L} - \{\bar{\mathbf{x}}\}$\\
\> \> set ${\bf AV} =$ non-fixed components of $\bar{\mathbf{x}}$\\
\> \> set ${\bf Prob} = {\bf PILP}|_{\bar{\mathbf{x}}}$\\ 
\> \> compute the set of projections associated to {\bf Prob}: $\big\{ {\bf P}_{\bf j} \big\}_{j \, : \, x_j \in {\bf AV}}$\\
\> \> {\bf goto} Step 1\\
\> {\bf else}\\
\> \> {\bf return} ${\bf CS}_z^{\bf PILP}$\\
\> {\bf end if}
\end{tabbing}
\end{center}
\rule[0.1cm]{10cm}{0.01cm}

\subsection{Correctness of the PSA-{\em pilp} algorithm} 
\label{subsec:4.2}

This section is intended to prove that the algorithm finds an optimal solution, or detects infeasibility, in a finite number of iterations. Before we come to the theorem we will enunciate two lemmas.

\begin{lemma} \label{lemma:1}
Let $\mathbf{\tilde x} = ({\tilde x}_1, \dots, {\tilde x}_n)$ be a feasible solution to {\bf PILP} (\ref{def:ilp}) such that $z(\mathbf{\tilde x}) = {\tilde z}$. Then, ${\tilde x}_j \in Range_j^{\tilde z}$ for all $j= 1, \dots, n$.
\end{lemma}

\begin{proof} 
The result follows from the definitions (\ref{def:range}), (\ref{def:proj}), (\ref{def:plow}) and (\ref{def:pup}), and from the fact that $\mathbf{\tilde x}$ is a feasible solution to {\bf PILP}. $\square$
\end{proof}

Note that, when the projections are restricted to the optimal level of the problem, say $z_{\bf PILP}$, lemma \ref{lemma:1} asserts that every optimal solution to {\bf PILP} can be reconstructed from the information provided by the sets $Range_j^{z_{\bf PILP}}$, $j = 1, \dots, n$.

\begin{lemma} \label{lemma:2}
Let $\mathbf{x}^* =(x_1^*, \dots, x_n^*)$ be an optimal solution to {\bf PILP} (\ref{def:ilp}), and let $\bar{\mathbf{x}}$ be the partial candidate solution defined by $\bar{\mathbf{x}} = (?, \dots, ?, x^*_k, \dots, x^*_n), \ k > 1$. Then, $\mathbf{x}^*$ is optimal to ${\bf PILP}|_{\bar{\mathbf{x}}}$. Furthermore, the problems {\bf PILP} and ${\bf PILP}|_{\bar{\mathbf{x}}}$ have both the same optimal objective value.
\end{lemma} 

\begin{proof} 
The results follow from the fact that $\mathbf{x}^*$ is feasible for both {\bf PILP} and ${\bf PILP}|_{\bar{\mathbf{x}}}$. $\square$
\end{proof}

\begin{theorem}
The PSA-{\em pilp} algorithm converges to an optimal solution, or detects infeasibility, in a finite number of steps.
\end{theorem}

\begin{proof} The finiteness of the algorithm follows from the fact that there is always a finite number of levels and a finite number of candidate solutions to be analysed. For this reason, the algorithm always stops after a finite number of iterations. In the case that the PILP problem is infeasible, the algorithm terminates returning this condition. 

To prove the correctness of the algorithm we need only to show that, if $\mathbf{x}^*=(x_1^*, \ldots, x^*_n)$ is an optimal solution to {\bf PILP} (\ref{def:ilp}) and $z_{\bf PILP}$ is the optimal objective value of the problem, then 
\[
\mathbf{x}^* \in {\bf CS}_{z_{\bf PILP}}^{\bf PILP} = \texttt{Inspect\_Level} \big( \ {\bf PILP}, z_{\bf PILP}, \big\{ {\bf P}_{\bf j} \big\}_{j=1, \dots, n} \ \big).
\]
We are going to prove this property by induction on the number of variables. 

For one-variable PILP problems the situation is as follows:
\begin{eqnarray*}
({\bf PILP}) \ \ \ \ \mathrm{maximize} \ \ z(x_1) = c_1x_1 + h \\
\mathrm{subject \ to} \ \ a_{11}x_1 \leq b_1\\
\vdots \hspace*{7 ex}\\
a_{m1}x_1\leq b_m\\
x_1 \in \Z_+ 
\end{eqnarray*}
We assume, without loss of generality, that $c_1 > 0$ and $h=0$. Let $[l_1,u_1], \ l_1, u_1 \in \R,$ be the feasible domain of the LP-relaxation of {\bf PILP}. Then, {\em the} optimal solution to {\bf PILP} is reached at $\mathbf{x}^* = (\lfloor u_1 \rfloor)$ yielding an objective value of $z_{\bf PILP}=c_1 \lfloor u_1 \rfloor$.
By applying the \texttt{Inspect\_Level} algorithm to {\bf PILP} restricted to level $z_{\bf PILP}$, it is easy to see that $Range_1^{z_{\bf PILP}} = \{ \lfloor u_1 \rfloor \}$. This implies ${\bf CS}_{z_{\bf PILP}}^{\bf PILP} = \{ \bar{\mathbf{x}}=(\lfloor u_1 \rfloor) \}$. Then, the theorem is true for every instance of a PILP problem with one variable.

Inductive step. Suppose that the result is valid for every PILP problem with $k$ variables, $k<n$. Let us now demonstrate that the property is also valid for every PILP problem with $n$ variables. Let {\bf PILP} be a PILP problem with $n$ variables satisfying (\ref{def:ilp}), and let $\mathbf{x}^*=(x_1^*, \ldots, x^*_n)$ be an optimal solution to {\bf PILP}. From lemma \ref{lemma:1}, it follows that $x^*_j \in Range_j^{z_{\bf PILP}} \ \forall j \in \{1,\ldots,n\}$. Then, by applying the \texttt{Inspect\_Level} algorithm to {\bf PILP} restricted to level $z_{\bf PILP}$, only one of the following alternatives holds: 

\begin{enumerate}
\item $\mathbf{\exists j \in \{1,\ldots,n\} \ / \ |Range_j^{z_{\bf PILP}}|=1}$. Let us suppose, without loss of generality, that $|Range_j^{z_{\bf PILP}}|~=~1$ $\forall j \in \{k,\ldots,n\}$ for some $k\geq 1$. That is, $Range_j^{z_{\bf PILP}}=\{x^*_j\} \ \forall j \in \{k,\ldots,n\}$.

\begin{enumerate}

\item If $k=1$, then $\mathbf{x}^* \in {\bf CS}_{z_{\bf PILP}}^{\bf PILP}$.

\item If $k>1$, the \texttt{Inspect\_Level} algorithm updates the partial candidate solution, the set of active variables, and the reduced problem in the following manner: 
\[
\bar{\mathbf{x}} = (?, \dots, ?, x^*_k, \dots, x^*_n), \ {\bf AV} = \{x_1, \dots, x_{k-1}\}, \ \mathrm{and} \ {\bf Prob}= {\bf PILP}|_{\bar{\mathbf{x}}}.
\]
It then recalculates $\big\{ {\bf P}_{\bf j} \big\}_{j \, : \, x_j \in {\bf AV}}$, and {\em restarts the process from step 1 until the algorithm ends}. We now observe that this last operation is equivalent to apply 
\[
{\bf CS}_{z_{\bf PILP}}^{\bf Prob} = \texttt{Inspect\_Level} \big( \ {\bf Prob}, z_{\bf PILP}, \big\{ {\bf P}_{\bf j} \big\}_{j \, : \, x_j \in {\bf AV}} \ \big),
\]
and then to extend the set of candidate solutions produced by the \texttt{Inspect\_Level} algorithm to a set of candidate solutions valid for {\bf PILP}. This operation is performed by setting $\bar{x}_j=x^*_j$ ($j = k,\ldots,n$) for all $\bar{\mathbf{x}} \in {\bf CS}_{z_{\bf PILP}}^{\bf Prob}$. 

To conclude, to prove that $\mathbf{x}^* \in {\bf CS}_{z_{\bf PILP}}^{\bf PILP}$, it suffices to show that $(x_1^*, \dots, x_{k-1}^*) \in {\bf CS}_{z_{\bf PILP}}^{\bf Prob}$. This result follows from lemma \ref{lemma:2} and the induction hypothesis. 

\end{enumerate}
\item $\mathbf{|Range_j^{z_{\bf PILP}}|>1 \ \forall j \in \{1,\ldots,n\}}$. Without loss of generality, let us consider that index $n$ is chosen. For each value $r_i \in Range_n^{z_{\bf PILP}}=\{r_1,\ldots,r_{|Range_n^{z_{\bf PILP}}|}\}$ a partial candidate solution is added to {\bf L} by setting $\bar{x}_n=r_i$. From lemma \ref{lemma:1}, it follows that $x^*_n \in Range_n^{z_{\bf PILP}}$, i.e., there exists $\bar{\mathbf{x}}^* \in {\bf L}$ such that $\bar{x}^*_n=x^*_n$. The algorithm then analyses all the partial candidate solutions added to {\bf L} and, therefore, after a finite number of steps it considers the candidate $\bar{\mathbf{x}}^*$, and defines ${\bf AV}=\{x_1, \dots, x_{n-1}\}$ and ${\bf Prob} = {\bf PILP}|_{\bar{\mathbf{x}}^*}$. Without loss of generality, let us suppose that $\bar{\mathbf{x}}^*$ is the only partial candidate contained in {\bf L} when it is chosen. The process is then restarted from step 1 until the algorithm terminates. The rest of the proof continues in the same manner as in case 1(b). $\square$
\end{enumerate}

\end{proof}


\section{Computational experiments}
\label{sec:5}

The performance of the PSA-{\em pilp} algorithm was compared with that of CPLEX v.12.1.0 (default) on two types of instances randomly generated of the 0-1MKP. Our algorithm was written in C, and the tests were carried out on one core of an Intel i7 3.40GHz with 16 GB of RAM.

\subsection{Data generation of test instances}
\label{subsec:5.1}

We consider the 0-1MKP, which is stated as follows:
\[
({\bf 0-1MKP}) \ \ \mathrm{max} \ \ z(\mathbf{x}) = \sum_{j = 1}^n c_jx_j \ \ \mathrm{s.t.} \ \ \sum_{j = 1}^n a_{ij}x_j \leq b_i, \ i \in \{1,2, \dots ,m\},
\]
with $\mathbf{x} = (x_1, \dots, x_n ) \in \{0,1\}^n$, and $c_j$, $a_{ij}$ and $b_i \in \Z_+$ $\forall j\in \{1,\ldots,n\}, \ i \in \{1,\ldots,m\}$. 

The test instances used in this section were randomly generated following the procedure proposed in Fr\'eville \cite{Fre}. In all of these instances the coefficients $a_{ij}$ are integer numbers uniformly generated in $U(0,1000)$; the right-hand side coefficients ($b_i$'s) are set using the formula $b_i = \alpha \sum_{j \in N} a_{ij}$, where $\alpha$ is the tightness ratio; and the objective function coefficients ($c_j$'s) are correlated to $a_{ij}$ as follows: 
\begin{itemize}
\item uncorrelated: $c_j \in U(0,1000)$
\item weakly correlated: $c_j = \frac{\sum_{i=1}^m a_{ij}}{m} + \xi$, \ with $\xi \in U(-100,100)$
\end{itemize}

The test instances were generated by varying combinations of constraints ($m = 3$ up to $5$) and variables (from $n = 200$ to $n = 10,000$). The tightness ratio, $\alpha$, was always fixed to $0.5$. For each $n - m$ combination, $5$ problems were generated.

\subsection{Implementation details of the PSA-{\em pilp} algorithm}
\label{subsec:5.2}

The implementation of the PSA-{\em pilp} algorithm that was used to carry out the computational experiments reported in this part of the paper presents the following characteristics. 

\begin{itemize}
\item The two-phase procedure described in Section \ref{subsec:3.3} was applied to determine the coefficients $P_j^{up}(0)$ and $P_j^{up}(1)$ for every instance of the 0-1MKP. In addition, the LP problems encountered during this routine were solved using the CPLEX callable library.
	
\item Due to the particular characteristics of the 0-1MKP, the lower projections of every test instance were calculated exactly by means of the following formula: $P_j^{low}(x_j) = c_j x_j + h$ for all $x_j \in [0,1]$. 

\item Every time the condition {\em $|Range_j^{z}| > 1$ for all $j$ such that $x_j \in {\bf AV}$} was reached, the active variable corresponding to the largest objective value was selected to split the partial candidate solution being scanned. 

\item The {\em last in, first out} strategy was employed to manage the list {\bf L} during the execution of the \texttt{Inspect\_Level} algorithm.
\end{itemize}

\subsection{Results and discussion}
\label{subsec:5.3}
 
Tables \ref{tab:1} and \ref{tab:2} below summarize the results obtained by both solvers on the two types of instances described previously. Columns {\bf CPLEX} and {\bf PSA-{\em pilp}} report the number of instances solved to optimality by each algorithm, followed by the average runtime (in CPU seconds) of those instances. If  the number of instances solved to optimality is less than $5$, this indicates that the algorithm failed because it ran out of memory when solving some of the instances. Column {\bf levels} shows the average number of levels scanned by the PSA-{\em pilp} algorithm until an optimal solution was reached. Column ${\bf AV}$ indicates the percentage of variables that remain {\em active} after the first iteration of the PSA-{\em pilp} algorithm at the optimal level. Column {\bf ratio} shows the average CPU time ratio between PSA-{\em pilp} and CPLEX for solving the given set of instances. Finally, column {\bf memory} indicates the average maximum virtual memory consumption (in megabytes) used by each algorithm (CPLEX/PSA-{\em pilp}) for solving the given set of instances. In all tests reported in this paper we did not limit the running time nor the memory consumption.

\begin{table}[!h]
\begin{center}
\caption{computational experiments on {\em Uncorrelated} instances}
\label{tab:1} 
\begin{tabular}{llllllll}
\hline\noalign{\smallskip}
{\bf n} & {\bf m} & {\bf CPLEX} & {\bf PSA-{\em pilp}} & {\bf levels} & {\bf AV} & {\bf ratio} & {\bf memory (CPLEX/PSA-{\em pilp})}\\ 
\noalign{\smallskip}\hline\noalign{\smallskip}
1,000 & 3 &  {\bf (5) 3 s.}  &  (5) 11 s. & 37.6 & 7.4 & 3.751 & negligible / negligible\\ 
5,000 & 3 &  {\bf (5) 109 s.}  &  (5) 180 s. & 13.8 & 2.6 & 1.648 & negligible / negligible\\
10,000 & 3 & {\bf (5) 830 s. }  &  (5) 1,992  s. & 7.2 & 1.4 & 2.39 & 953 / 283 \\
\noalign{\smallskip}\hline\noalign{\smallskip}
1,000 & 4 &  {\bf (5) 27 s.}  &  (5) 212 s. & 46.8 & 9.5 & 7.754 & negligible / negligible\\
3,000 & 4 &  {\bf (5) 1,473 s.}  &  (5) 4,959 s. & 27.4 & 5.2 & 3.36 & 1,001 / 411 \\
5,000 & 4 &  {\bf (5) 5,358 s.}  &  (5) 9,702 s. & 18.6 & 3.8 & 1.81 & 3,233 / 644\\
10,000 & 4 &  (5) 39,957 s.  & {\bf (5) 32,619 s.} & 11.2 & 2.3 & 0.81 & 22,193 / 1,765 \\
\noalign{\smallskip}\hline\noalign{\smallskip}
1,000 & 5 &  {\bf (5) 108 s.}  &  (5) 1,063 s. & 67.6 & 14.2 & 9.82 & negligible / negligible\\ 
3,000 & 5 &  {\bf (5) 10,044 s.}  &  (5) 38,706 s. & 35.2 & 7.1 & 3.85 & 4,595 / 1,826\\ 
\noalign{\smallskip}\hline
\end{tabular}
\end{center}
\end{table}

\begin{table}[!h]
\begin{center}
\caption{computational experiments on {\em Weakly Correlated} instances}
\label{tab:2}
\begin{tabular}{llllllll}
\hline\noalign{\smallskip}
{\bf n} & {\bf m} & {\bf CPLEX} & {\bf PSA-{\em pilp}} & {\bf levels} & {\bf AV} & {\bf ratio} & {\bf memory (CPLEX/PSA-{\em pilp})}\\ 
\noalign{\smallskip}\hline\noalign{\smallskip}
2,000 & 3  &  (5) 792 s.  & {\bf (5) 593 s.} &  8.6 & 8.1 & 0.749 &  620 / 213\\
3,000 & 3  &  (5) 1,311 s.  & {\bf (5) 697 s.} & 6.8 & 5.9 & 0.53 & 1,379 / 123  \\ 
5,000 & 3  &  (5) 3,704 s.  & {\bf (5) 1,413 s.} & 4 & 3.4 & 0.381 & 4,369 / 185\\ 
10,000 & 3  &  (5) 5,226 s.  & {\bf (5) 2,971 s.} & 2.8 & 2.2 & 0.568 & 6,709 / 318\\ 
\noalign{\smallskip}\hline\noalign{\smallskip}
200 & 4  &  {\bf (5) 35 s.}  &  (5) 228 s. & 56.6 & 56.8 & 6.551 & negligible / negligible\\
500 & 4  &  {\bf (5) 355 s.}  &  (5) 988 s. & 29.4 & 30 & 2.783 & negligible / negligible\\ 
1,000 & 4  &  {\bf (5) 5,567 s.}  &  (5) 8,332 s. & 20.8 & 20.4 & 1.49 & 3,953 / 905 \\  
2,000 & 4  &  (5) 32,214 s.  & {\bf (5) 19,503 s.} & 11.2 & 11.1 & 0.60 & 24,509 / 2,035 \\ 
3,000 & 4  &  (2) $>$59,931 s.  & {\bf (5) 58,529 s.} &  8.6 & 8.3 & $<$0.976 & $>$63,658 / 3,502 \\ 
\noalign{\smallskip}\hline\noalign{\smallskip}
200 & 5  &  {\bf (5) 202 s.}  &  (5) 1,821 s. & 69 & 70 & 9.021 & negligible / negligible\\ 
500 & 5  &  {\bf (5) 9,234 s.}  &  (5) 29,591 s. & 37.2 & 38 & 3.20 & 4,955 / 2,036 \\ 
1,000 & 5  &  {\bf (5) 89,542 s.}  &  (5) 93,581 s. & 25.8 & 26.2 & 1.04 & 41,244 / 4,935\\ 
\noalign{\smallskip}\hline
\end{tabular}
\end{center}
\end{table}

Based on the computational results, we conclude that the PSA-{\em pilp} algorithm is not very efficient, in terms of running time, to solve small-size instances; however, it shows a better trend than CPLEX (see ratio) when the number of variables increases, especially in the hardest type of instances. In this regard, it is worth noting that, in contrast to CPLEX (default), the implementation of the PSA-{\em pilp} algorithm does not incorporate any type of presolve, cutting plane technique, or heuristics to improve its performance. Concerning memory usage, the numbers of the PSA-{\em pilp} algorithm are considerably lower than those of CPLEX in all instances tested. The PSA-{\em pilp} algorithm consumed in average less than 10.4\% of the memory consumed by CPLEX. This can be explained by the way the algorithm conducts the search process for the optimal solution (by generating candidate solutions tailored to specific values of the objective function), and by the manner in which the {\bf L} set is managed during the execution of the \texttt{Inspect\_Level} algorithm. In fact, under these conditions it can be proven that PSA-{\em pilp}'s memory consumption is polynomial in the number of variables and the cardinality of the sets $Range_j^{z}$. Finally, it is interesting to note that the percentage of variables that are fixed to their optimal value in the first iteration of the algorithm at the optimal level, grows to more than 97\% of the total variables.


\section{Conclusions and future work} 
\label{sec:6}

This paper proposes a new exact algorithm, called PSA-{\em pilp}, for solving PILP problems using projections. The PSA-{\em pilp} algorithm differs from state-of-the-art techniques since it searches for solutions for specific values of the objective function. As a consequence of this approach, the number of variables in the original problem is systematically reduced (for each considered level) and no additional constraints are added to the initial formulation. According to our computational experiments, we believe that the proposed new algorithm paradigm has a great potential as a useful tool for solving PILP problems.

The present work leaves open a number of interesting directions for future research. First, the current version of the PSA-{\em pilp} algorithm could be greatly improved through the incorporation of advanced search strategies, preprocessing and probing techniques, cutting plane algorithms, and primal heuristics. 
Second, additional improvements can be reached via parallel computing techniques. Projection-splitting-based algorithms are natural candidates for parallelization because the subproblems associated with each level and each partial candidate solution contained in {\bf L} are completely independent. Thus, parallelism can be exploited by evaluating multiple levels and multiple partial candidate solutions simultaneously. 
Finally, it is relatively easy to see how the proposed methodology can be extended to more complex situations such as PILP problems in which the condition {\em $\mathbf{c} \in \Z^n$ and $h \in \Z$} in the objective function is relaxed, or even to MILP problems. We are going to deal with this discussion in the second part of this series.\\

\begin{acknowledgements}
This work was partially supported by grants UBACYT 20020100100666, PICT 2010-304, PICT 2011-817.
We thank Luis Mastrangelo and Santiago Feldman for their helpful suggestions and constructive criticisms.
\end{acknowledgements}


\end{document}